\documentclass[10pt, conference]{IEEEtran}
\usepackage[english]{babel }

\usepackage[margin=54pt]{geometry}

\title{\vspace*{18pt}  Control of MTDC Transmission Systems under~Local~Information}

\author{ \IEEEauthorblockA{Martin Andreasson$^{\IEEEauthorrefmark{2}}$, Dimos V. Dimarogonas, Henrik Sandberg and  Karl H. Johansson  
}
\\
\thanks{This work was supported in part by the European Commission by the Hycon2 project, the Swedish Research Council (VR) and the Knut and Alice Wallenberg Foundation. 
We would like to thank the anonymous reviewers for their encouraging and insightful feedback. Their comments have helped to improved the presentation of the paper significantly. 
The authors are with the ACCESS Linnaeus Centre, KTH Royal Institute of Technology, Stockholm, Sweden.The $2^{\text{nd}}$ author is also affiliated with the Centre for Autonomous Systems at KTH.
\IEEEauthorrefmark{2} Corresponding author. E-mail: mandreas@kth.se}
}
\IEEEoverridecommandlockouts

\usepackage{mathrsfs}
\usepackage{amsfonts}
\usepackage{amssymb}
\usepackage{amsmath}
\usepackage{graphicx}
\usepackage{amsthm}
\usepackage{commath}
\usepackage{tikz}
\usepackage{tkz-graph}
\usetikzlibrary{arrows}
\usepackage{flushend}
\usepackage{pgfplots}
\usepackage{todonotes}
\usetikzlibrary{external}
\newtheorem{theorem}{Theorem}

\newtheorem{remark}{Remark}
\newtheorem{objective}{Objective}

\DeclareMathOperator*{\argmin}{argmin}

\DeclareMathOperator*{\diag}{diag}

\newcommand{\beq}{\begin{equation}}
\newcommand{\eeq}{\end{equation}}
\newcommand{\bq}{\begin{eqnarray}}
\newcommand{\eq}{\end{eqnarray}}
\newcommand{\bqn}{\begin{eqnarray*}}
\newcommand{\eqn}{\end{eqnarray*}}
\newcommand{\bee}{\begin{enumerate}}
\newcommand{\eee}{\end{enumerate}}

\newlength\fheight
\newlength\fwidth

\begin{document}
\maketitle

\begin{abstract}
High-voltage direct current (HVDC) is a commonly used technology for long-distance electric power transmission, mainly due to its low resistive losses. 
In this paper a distributed controller for multi-terminal high-voltage direct current (MTDC) transmission systems is considered. Sufficient conditions for when the proposed controller renders the closed-loop system asymptotically stable are provided. Provided that the closed loop system is asymptotically stable, it is shown that in steady-state a weighted average of the deviations from the nominal voltages is zero. Furthermore, a quadratic cost of the current injections is minimized asymptotically.
\end{abstract}

\section{Introduction}

Transmitting power over long distances while minimizing losses is one of the greatest challenges in today's power transmission systems. Increased distances between power generation and consumption is a driving factor behind long-distance power transmission. One such example are large-scale off-shore wind farms, which often require power to be transmitted in cables over long distances to the mainland power grid. High-voltage direct current (HVDC) power transmission is a commonly used technology for long-distance power transmission. Its higher investment costs compared to AC transmission lines are compensated by its lower resistive losses for sufficiently long distances. The break-even point, i.e., the point where the total costs of overhead HVDC and AC lines are equal, is typically 500-800 km \cite{padiyar1990hvdc}. However, for cables, the break-even point is typically lower than 100 km \cite{bresesti2007hvdc}. Increased use of HVDC for electrical power transmission suggests that future HVDC transmission systems are likely to consist of multiple terminals connected by several HVDC transmission lines. Such systems are referred to as Multi-terminal HVDC (MTDC) systems in the literature \cite{van2010multi}. 

Maintaining an adequate DC voltage is one of the most important  control problems for HVDC transmission systems. Firstly, the voltage levels at the DC buses govern the current flows by Ohm's law and Kirchhoff's circuit laws. Secondly, if the DC voltage deviates too far from a nominal operational voltage, equipment could be damaged, resulting in loss of power transmission capability \cite{van2010multi}.
 
Different voltage control methods for HVDC systems have been proposed in the literature. Among them, the voltage margin method (VMM) and the voltage droop method (VDM) are the most well-known methods \cite{haileselassie2009control}.
These voltage control methods change the active injected power to maintain active power balance in the DC grid and as a consequence, control the DC voltage. A decreasing DC voltage requires increased injected currents through the DC buses in order to restore the voltage.

The VDM controller is designed so that several DC buses participate to control the DC voltage through proportional control \cite{14}. All participating terminals change their injected active power to a level proportional to the deviation from the nominal voltage \cite{haileselassie2009control, zonetti2014modeling}. These decentralized proportional controllers induce static errors in the voltage, which is the main disadvantage of VDM. 
 
The VMM controller on the other hand, is designed so that one terminal is responsible to control the DC voltage, by e.g., a PI controller. The other terminals keep their injected active power constant. The terminal controlling the DC voltage is called a slack terminal. When the slack terminal is no longer able to supply or extract the power necessary to maintain its DC bus voltage within a certain threshold, a new terminal will operate as the slack terminal \cite{8}. The transition between the slack terminals can cause conflicts between the controllers, and requires one or a few terminals to inject all the current needed to maintain an adequate voltage \cite{8}.


Distributed control has been successfully applied to both primary and secondary frequency control of AC transmission systems \cite{andreasson2013distributed, simpson2012synchronization, li2014connecting, Andreasson2014TAC}. Recently, distributed controllers have been applied also to  secondary frequency control of asynchronous AC transmission systems connected through an MTDC system \cite{dai2010impact}. In \cite{Andreasson2014_IFAC}, a distributed controller for voltage control of MTDC systems was proposed. It was shown that the controller can regulate the voltages of the terminals, while the injected power is shared fairly among the DC buses. However, this controller possesses the disadvantage of requiring a terminal dedicated to measuring and controlling the voltage. 
 In this paper, we propose a fully distributed voltage controller for MTDC transmission systems, which possesses the property of fair power sharing, asymptotically minimizing the cost of the power injections.

 The remainder of this paper is organized as follows. In Section \ref{sec:prel}, the mathematical notation is defined. In Section \ref{sec:model}, the system model and the control objectives are defined. In Section \ref{sec:dist_control}, a distributed averaging controller is presented, and its stability and steady-state properties are analyzed. In Section \ref{sec:simulations}, simulations of the distributed controller on a four-terminal MTDC test system are provided, showing the effectiveness of the proposed controller. The paper ends with a discussion and concluding remarks in Section \ref{sec:discussion}.

\section{Notation}
\label{sec:prel}
Let $\mathcal{G}$ be a graph. Denote by $\mathcal{V}=\{ 1,\hdots, n \}$ the vertex set of $\mathcal{G}$, and by $\mathcal{E}=\{ 1,\hdots, m \}$ the edge set of $\mathcal{G}$. Let $\mathcal{N}_i$ be the set of neighboring vertices to $i \in \mathcal{V}$.
In this paper we will only consider static, undirected and connected graphs. For the application of control of MTDC power transmission systems, this is a reasonable assumption as long as there are no power line failures.
Denote by $\mathcal{B}$ the vertex-edge adjacency matrix of a graph, and let $\mathcal{\mathcal{L}_W}=\mathcal{B}W\mathcal{B}^T$ be its weighted Laplacian matrix, with edge-weights given by the  elements of the positive definite diagonal matrix $W$.
Let $\mathbb{C}^-$ denote the open left half complex plane, and $\bar{\mathbb{C}}^-$ its closure. We denote by $c_{n\times m}$ a matrix of dimension $n\times m$ whose elements are all equal to $c$, and by $c_n$ a column vector whose elements are all equal to $c$. For a symmetric matrix $A$, $A>0 \;(A\ge 0)$ is used to denote that $A$ is positive (semi) definite. $I_{n}$ denotes the identity matrix of dimension $n$. For simplicity, we will often drop the notion of time dependence of variables, i.e., $x(t)$ will be denoted $x$ for simplicity.

\section{Model and problem setup}
\label{sec:model}
Consider a MTDC transmission system consisting of $n$ DC buses, denoted by the vertex set $ \mathcal{V}= \{1, \dots, n\}$, see Figure~\ref{fig:graph} for an example of an MTDC topology. The DC buses are modelled as ideal current sources which are connected by $m$ HVDC transmission lines, denoted by the edge set $ \mathcal{E}= \{1, \dots, m\}$. The dynamics of any system (e.g., an AC transmission system) connected through the DC buses, or any dynamics of the DC buses (e.g., AC-DC converters) are neglected. The HVDC lines are assumed to be purely resistive, implying that 
\begin{align*}
I_{ij} = \frac{1}{R_{ij}} (V_i -V_j),
\end{align*}
due to Ohm's law, where $V_i$ is the voltage of bus $i$, $R_{ij}$ is the resistance and $I_{ij}$ is current of the HVDC line from bus $i$ to $j$. 
The voltage dynamics of an arbitrary DC bus $i$ are thus given by
\begin{align}
\begin{aligned}
C_i \dot{V}_i &= -\sum_{j\in \mathcal{N}_i} I_{ij} + I_i^{\text{inj}} + u_i \\
&= -\sum_{j\in \mathcal{N}_i} \frac{1}{R_{ij}}(V_i -V_j) + I_i^{\text{inj}} + u_i,
\end{aligned}
\label{eq:hvdc_scalar}
\end{align}
where $C_i$ is the total capacity of bus $i$, including shunt capacities and the capacitance of the HVDC line, $I_i^{\text{inj}}$ is the nominal injected current, which is assumed to be unknown but constant over time, and $u_i$ is the controlled injected current. Equation \eqref{eq:hvdc_scalar} may be written in vector-form as
\begin{align}
\begin{aligned}
\dot{V} &= -C\mathcal{L}_R V +CI^{\text{inj}} + Cu,
\end{aligned}
\label{eq:hvdc_vector}
\end{align}
where $V=[V_1, \dots, V_n]^T$, $C=\diag([C_1^{-1}, \dots, C_n^{-1}])$, $I^{\text{inj}} = [I^{\text{inj}}_1, \dots, I^{\text{inj}}_n]^T$, $u=[u_1, \dots, u_n]^T$ and $\mathcal{L}_R$ is the weighted Laplacian matrix of the graph representing the transmission lines, whose edge-weights are given by the conductances $\frac{1}{R_{ij}}$. The control objectives considered in this paper are twofold.

\begin{figure}
\begin{center}
\begin{tikzpicture}
\GraphInit[vstyle=Normal]
\Vertex[x=0, y=0] {1}
\Vertex[x=2, y=0] {2}
\Vertex[x=0, y=-1.5] {3}
\Vertex[x=2, y=-1.5] {4}
\Edge[label=$e_1$](1)(2)
\Edge[label=$e_4$](3)(4)
\Edge[label=$e_2$](1)(3)
\Edge[label=$e_3$](2)(4)
\end{tikzpicture}
\end{center}
\caption{Example of a graph topology of an MTDC system.}
\label{fig:graph}
\end{figure}
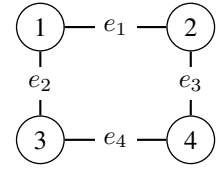

\begin{objective}
\label{obj:1}
The voltages of any DC bus, $V_i$, should converge to a value close to the nominal voltage for bus $i$ ($V^{\text{nom}}_i$), after a disturbance has occurred. More precisely, a weighted average of the steady-state errors should be zero:
\begin{align*}
\lim_{t\rightarrow \infty} \sum_{i=1}^n K_i^V \left(V(t) - V_i^{\text{nom}}\right)= 0 ,
\end{align*}
for some $K^V_i>0, i=1, \dots, n$. Furthermore, the asymptotic voltage differences between the DC buses should be bounded, i.e., $\lim_{t\rightarrow \infty} |V_i(t)-V_i(t)|\le V^* \; \forall i,j \in \mathcal{V}$, for some $V^*>0$.
\end{objective}

\begin{remark}
It is in general not possible to have $\lim_{t\rightarrow \infty}V_i(t) = V^{\text{nom}}_i$ for all $i \in \mathcal{V}$, since this by Ohm's law would imply that the HVDC line currents are always unchanged, not allowing for time-varying demand. 
\end{remark}

\begin{objective}
\label{obj:2}
The cost of the current injections should be minimized asymptotically. More precisely, we require
\begin{align*}
\lim_{t\rightarrow \infty} u(t) = u^*,
\end{align*}
where $u^*$ is defined by
\begin{align}
[u^*,V^*]=\argmin_{[u,V]} \sum_{i \in \mathcal{V}} \frac{1}{2} f_i u_i^2 \quad \text{s.t.} \quad \mathcal{L}_R V  &= I^{\text{inj}} + u, \label{eq:opt1}
\end{align}
and where $f_i>0, i=1, \dots , n$ are positive constants. 
\end{objective}

\begin{remark}
Objective \ref{obj:2} is analogous to the quadratic optimization of AC power generation costs considered in \cite{andreasson2013distributed, simpson2012synchronization}. 
\end{remark}

\section{Distributed MTDC control}
\label{sec:dist_control}
It was shown in \cite{Andreasson2014_IFAC} that a decentralized proportional droop controller cannot satisfy Objective \ref{obj:1} and \ref{obj:2} simultaneously. Furthermore, a proportional controller can only satisfy Objective \ref{obj:1} or \ref{obj:2} if the proportional gains tend  to infinity or $0$, respectively. A distributed controller was proposed, which was shown to satisfy Objective \ref{obj:1} and \ref{obj:2} simultaneously. However, this controller requires one specific DC bus to measure and control the voltage. This controller thus has the disadvantage of being sensitive to failure of this specific terminal. 
In this section we propose a novel, fully distributed controller for MTDC networks which allows for communication between the buses. This controller does not rely on a single leader, but the voltage regulation is distributed among all buses. The proposed controller takes inspiration from the control algorithms given in \cite{andreasson2013distributed, Andreasson2014_IFAC, simpson2012synchronization}, and is given by
\begin{align}
\label{eq:distributed_voltage_control}
\begin{aligned}
u_i &= -K^P_i(V_i - \hat{V}_i - \bar{V}_i ) \\
\dot{\hat{V}}_i &= -\gamma \sum_{j\in \mathcal{N}_i} c_{ij} \left( (\hat{V}_i + \bar{V}_i -V_i){-}(\hat{V}_j + \bar{V}_j -V_j) \right) \\
\dot{\bar{V}}_i &= - K_i^V (V_i - V_i^{\text{nom}}) -\delta \sum_{j\in \mathcal{N}_i} c_{ij} (\bar{V}_i - \bar{V}_j).
\end{aligned}
\end{align}
The first line of the controller \eqref{eq:distributed_voltage_control} can be interpreted as a proportional controller, whose reference value is controlled by the remaining two lines. The second line ensures that the weighted current injections converge to the identical optimal value through a consensus-filter. The third line is a distributed secondary voltage controller, where each terminal measures the voltage and updates the reference value through a consensus-filter. 
 In vector-form, \eqref{eq:distributed_voltage_control} can be written as
\begin{align}
\label{eq:distributed_voltage_control_vector}
\begin{aligned}
u &= -K^P (V - \hat{V} - \bar{V} ) \\
\dot{\hat{V}} &= -\gamma \mathcal{L}_c (\hat{V} + \bar{V} -V) \\
\dot{\bar{V}} &= -K^V(V - V^{\text{nom}}) - \delta \mathcal{L}_c \bar{V},
\end{aligned}
\end{align}
where $K^P=\diag([K^P_1, \dots , K^P_n])$, $K^V=\diag([K^V_1, \dots, K^V_n])$, $V^{\text{nom}}=[V^{\text{nom}}_1, \dots , V^{\text{nom}}_n]^T$ and $\mathcal{L}_C$ is the weighted Laplacian matrix of the graph representing the communication topology, denoted $\mathcal{G}_c$, whose edge-weights are given by $c_{ij}$, and which is assumed to be connected.
Substituting the controller \eqref{eq:distributed_voltage_control_vector} in the system dynamics \eqref{eq:hvdc_vector}, yields
\begin{align}
\begin{aligned}
\begin{bmatrix}
\dot{\bar{V}} \\ \dot{\hat{V}} \\ \dot{V}
\end{bmatrix}
&=
\underbrace{
\begin{bmatrix}
-\delta \mathcal{L}_C & 0_{n\times n} & -K^V \\
-\gamma \mathcal{L}_C & -\gamma \mathcal{L}_C & \gamma \mathcal{L}_C \\
CK^P & CK^P & -C(\mathcal{L}_R + K^P)
\end{bmatrix}
\begin{bmatrix}
{\bar{V}} \\ {\hat{V}} \\ {V}
\end{bmatrix}}_{\triangleq A}
\\
&\;\;\;\;+ \underbrace{\begin{bmatrix}
K^VV^{\text{nom}} \\ 0_{n} \\ CI^{\text{inj}}
\end{bmatrix}}_{\triangleq b}.
\end{aligned}
\label{eq:cl_dynamics_vector}
\end{align}
 The following theorem characterizes when the controller \eqref{eq:distributed_voltage_control} stabilizes the system \eqref{eq:hvdc_scalar}, and shows that it has some desirable properties. 

\begin{theorem}
\label{th:distributed_voltage_control}
Consider an MTDC network described by \eqref{eq:hvdc_scalar}, where the control input $u_i$ is given by \eqref{eq:distributed_voltage_control} and the injected currents $I^{\text{inj}}$ are constant. Let $K^P=F^{-1}$, where $F=\diag([f_1, \dots, f_n])$. It is easily shown that  $A$ as defined in \eqref{eq:cl_dynamics_vector}, has one eigenvalue equal to $0$. If all other eigenvalues lie in the open complex left half plane, then:
\begin{enumerate}
\item $\lim_{t\rightarrow \infty} \sum_{i=1}^n K_i^V \left(V(t) - V_i^{\text{nom}}\right)= 0$ 
\item $\lim_{t\rightarrow  \infty } u(t) = u^*$, where $u^*$ is defined as in Objective \ref{obj:2}.
\end{enumerate}
 The relative voltage differences are also bounded and satisfy $\lim_{t\rightarrow \infty} |V_i(t)-V_i(t)|\le 2I^{\text{max}}\sum_{i=2}^n \frac{1}{\lambda_i}\; \forall i,j \in \mathcal{V}$, where $I^{\text{max}}=\max_i |I^{\text{tot}}|$ and $I^{\text{tot}}=  I^{\text{inj}} + \lim_{t\rightarrow \infty} u(t)$, and $\lambda_i$ denotes the $i$'th eigenvalue of $\mathcal{L}_R$. 
\end{theorem}

\begin{proof}
It is easily verified that the right-eigenvector of $A$ corresponding to the zero eigenvalue is $v_1={1}/{\sqrt{2n}}[1_{n}^T, -1_{n}^T, 0_{n}^T]^T$. Since $b$ as defined in \eqref{eq:cl_dynamics_vector}, is not parallel to $v_1$, $\lim_{t\rightarrow \infty} [{\bar{V}}(t), {\hat{V}}(t), {V}(t)]$ exists and is finite, by the assumption that all other eigenvalues lie in the open complex left half plane. Hence, we consider any stationary solution of \eqref{eq:cl_dynamics_vector}
\begin{align}
\begin{aligned}
\begin{bmatrix}
0_{n} \\ 0_{n} \\ 0_{n}
\end{bmatrix}
&=
{
\begin{bmatrix}
-\delta \mathcal{L}_C & 0_{n\times n} & -K^V \\
-\gamma \mathcal{L}_C & -\gamma \mathcal{L}_C & \gamma \mathcal{L}_C \\
CK^P & CK^P & -C(\mathcal{L}_R + K^P)
\end{bmatrix}
\begin{bmatrix}
{\bar{V}} \\ {\hat{V}} \\ {V}
\end{bmatrix}}
\\
&\;\;\;\;+ {\begin{bmatrix}
K^VV^{\text{nom}} \\ 0_{n} \\ CI^{\text{inj}}
\end{bmatrix}}.
\end{aligned}
\label{eq:cl_dynamics_vector_equilibrium}
\end{align}
Premultiplying \eqref{eq:cl_dynamics_vector_equilibrium} with $[1_{n}^T, 0_{n}^T, 0_{n}^T]$ yields
\begin{align*}
1_{n}^TK^V ( V^{\text{nom}}-V) = -\sum_{i=1}^n K_i^V V(t) + \sum_{i=1}^n K_i^V V_i^{\text{nom}}.
\end{align*}
The $n+1$:th to $2n$:th lines of \eqref{eq:cl_dynamics_vector_equilibrium} imply
\begin{align*}
\mathcal{L}_C(\bar{V}+\hat{V}-V) &= 0_{n} \Rightarrow \\
(\bar{V}+\hat{V}-V) &= k_1 1_{n} \Rightarrow \\
u = K^P (\bar{V}+\hat{V}-V) &=  k_1 K^P 1_{n}
\end{align*}
Now finally, premultiplying  \eqref{eq:cl_dynamics_vector_equilibrium} with $[0_{n}^T, 0_{n}^T, 1_{n}^TC^{-1}]$ yields
\begin{align*}
&1_{n}^T \left(K^P(\bar{V}+\hat{V}-V) + I^{\text{inj}} \right) = 1_{n}^T \left( k_1 K^P 1_{n} + I^{\text{inj}} \right) \\
&= k_1 \sum_{i=1}^n K^P_i + \sum_{i=1}^n I^{\text{inj}}_i=0_n,
\end{align*}
which implies $k_1= -\left( \sum_{i=1}^n I^{\text{inj}}_i \right) / \left( \sum_{i=1}^n K^P_i \right)$. The bound on $\lim_{t\rightarrow \infty} |V_i(t)-V_i(t)|$ follows from the proof of Theorem 3 in \cite{Andreasson2014_IFAC}. Since $K^P=F^{-1}$, any stationary solution of \eqref{eq:cl_dynamics_vector} satisfies $u=k_1F^{-1}1_{n}$. On the other hand, the KKT condition for the optimization problem \eqref{eq:opt1} is $Fu=\lambda 1_{n}$. Since \eqref{eq:opt1} is convex, the KKT condition is necessary and sufficient. This implies that any stationary solution of \eqref{eq:cl_dynamics_vector} solves \eqref{eq:opt1}. 
\end{proof}
While Theorem \ref{th:distributed_voltage_control} establishes an exact condition when the distributed controller \eqref{eq:distributed_voltage_control} stabilizes the MTDC system \eqref{eq:hvdc_scalar}, it does not give any insight in how to choose the controller parameters to achieve a stable closed loop system. The following theorem gives a sufficient stability condition for a special case.
\begin{theorem}
\label{th:stability_A}
The matrix $A$ as defined in \eqref{eq:cl_dynamics_vector}, always has one eigenvalue equal to $0$.
Assume that $\mathcal{L}_C = \mathcal{L}_R$, i.e. that the topology of the communication network is identical to the topology of the MTDC system. Assume furthermore that $K^P = k^P I_n$, i.e. the controller gains are equal. 
 Then the remaining eigenvalues lie in the open complex left half plane if
 \begin{align}
 &\frac{\gamma+\delta}{2k^P} \lambda_{\min}\left( \mathcal{L}_RC^{-1}  +C^{-1}\mathcal{L}_R \right) + 1 >0 \label{eq:stability_A_sufficient_1} \\
  &\frac{\gamma\delta}{2k^P} \lambda_{\min}\left( \mathcal{L}^2_RC^{-1} + C^{-1}\mathcal{L}_R^2 \right) + \min_i K^V_i >0 \label{eq:stability_A_sufficient_2} \\
  & \begin{aligned}
  &\lambda_{\max}\left(\mathcal{L}_R^3\right) \frac{\gamma\delta}{k^{P^2}} \\
  &\le \left( \frac{\gamma+\delta}{2k^P} \lambda_{\min}\left( \mathcal{L}_RC^{-1}  +C^{-1}\mathcal{L}_R \right) + 1 \right) \\
   &\left( \frac{\gamma\delta}{2k^P} \lambda_{\min}\left( \mathcal{L}^2_RC^{-1} + C^{-1}\mathcal{L}_R^2 \right) + \min_i K^V_i \right)
\end{aligned}   
    \label{eq:stability_A_sufficient_3}
 \end{align}
\end{theorem}
\begin{remark}
By choosing $\gamma$ and $\delta$ sufficiently small, and choosing $k^P$ and $\min_iK^V_i$ sufficiently large, the inequalities \eqref{eq:stability_A_sufficient_1}--\eqref{eq:stability_A_sufficient_3} can always be satisfied. Intuitively, this implies that the consensus dynamics in the network should be sufficiently slow compared to the voltage dynamics. 
\end{remark}
\begin{proof}[Proof of Theorem \ref{th:stability_A}]
\begin{figure*}
\begin{align}
\begin{aligned}
0&=\det(sI_{3n}-A) = \left| \begin{matrix}
sI_n +\delta \mathcal{L}_C & 0_{n\times n} & K^V \\
\gamma \mathcal{L}_C & sI_n + \gamma \mathcal{L}_C & -\gamma \mathcal{L}_C \\
-CK^P & -CK^P & SI_n + C(\mathcal{L}_R{+}K^P) 
\end{matrix} \right| 
= \left| \begin{matrix}
sI_n +\delta \mathcal{L}_C & 0_{n\times n} & K^V \\
-sI_n & sI_n + \gamma \mathcal{L}_C & -\gamma \mathcal{L}_C \\
0_{n\times n} & -CK^P & SI_n + C(\mathcal{L}_R{+}K^P) 
\end{matrix} \right| \\
&= {s}^n \left| \begin{matrix}
sI_n +\delta \mathcal{L}_C & 0_{n\times n} & K^V \\
-I_n & I_n + \frac{\gamma}{s} \mathcal{L}_C & -\frac{\gamma}{s} \mathcal{L}_C \\
0_{n\times n} & -CK^P & SI_n + C(\mathcal{L}_R{+}K^P) 
\end{matrix} \right| \\
 &= {s}^n \left| sI + \delta\mathcal{L}_C \right|^{-1} \left| \begin{matrix}
sI_n +\delta \mathcal{L}_C & 0_{n\times n} & K^V \\
-sI - \delta\mathcal{L}_C & sI_n + \gamma\mathcal{L}_C + \delta \mathcal{L}_C + \frac{\gamma\delta}{s} \mathcal{L}_C^2 & -\gamma\mathcal{L}_C +\frac{\gamma\delta}{s}\mathcal{L}_C^2 \\
0_{n\times n} & -CK^P & SI_n + C(\mathcal{L}_R{+}K^P) 
\end{matrix} \right| \\
&= {s}^n \left| sI + \delta\mathcal{L}_C \right|^{-1} \left| \begin{matrix}
sI_n +\delta \mathcal{L}_C & 0_{n\times n} & K^V \\
0_{n\times n} & sI_n + \gamma\mathcal{L}_C + \delta \mathcal{L}_C + \frac{\gamma\delta}{s} \mathcal{L}_C^2 & -\gamma\mathcal{L}_C -\frac{\gamma\delta}{s}\mathcal{L}_C^2 +K^V \\
0_{n\times n} & -CK^P & SI_n + C(\mathcal{L}_R{+}K^P) 
\end{matrix} \right| \\
&= s^n 
\left| \begin{matrix}
 sI_n + \gamma\mathcal{L}_C + \delta \mathcal{L}_C + \frac{\gamma\delta}{s} \mathcal{L}_C^2 & -\gamma\mathcal{L}_C -\frac{\gamma\delta}{s}\mathcal{L}_C^2 +K^V \\
 -CK^P & SI_n + C(\mathcal{L}_R{+}K^P) 
\end{matrix} \right| \\
&= s^n \left| SI_n + C(\mathcal{L}_R{+}K^P \right| \left|CK^P \right|
\left| \begin{matrix}
\begin{matrix}
\left(sI_n + \gamma\mathcal{L}_C + \delta \mathcal{L}_C + \frac{\gamma\delta}{s} \mathcal{L}_C^2 \right)\cdot \\
K^{P^{-1}}C^{-1}\left( sI_n + C(\mathcal{L}_R{+}K^P) \right)
\end{matrix} & -\gamma\mathcal{L}_C -\frac{\gamma\delta}{s}\mathcal{L}_C^2 +K^V \\
 -SI_n - C(\mathcal{L}_R{+}K^P) & SI_n + C(\mathcal{L}_R{+}K^P) 
\end{matrix} \right| \\
&=  \left|CK^P \right| \left| \left[ \gamma\delta \mathcal{L}_C^2K^{P^{-1}} \mathcal{L}_R \right] + s \left[ (\delta+\gamma)\mathcal{L}_CK^{P^{-1}}\mathcal{L}_R + \delta\mathcal{L}_C +\gamma\delta\mathcal{L}_C^2K^{P^{-1}}C^{-1} +K^V \right] \right. \\
& \;\;\;\; \left.  + s^2\left[ K^{P^{-1}}\mathcal{L}_R + I_n +(\gamma+\delta)\mathcal{L}_CK^{P^{-1}}C^{-1} \right] + s^3\left[ K^{P^{-1}}C^{-1} \right] \right| \triangleq \left|CK^P \right| \det\left(Q(s)\right)
\end{aligned}
\label{eq:char_eq_A}
\end{align}
\noindent\makebox[\linewidth]{\rule{\textwidth}{0.4pt}}
\end{figure*}
The characteristic equation of $A$ is given by equation \eqref{eq:char_eq_A}. Clearly, this equation has a solution only if $x^TQ(s)x=0$ has a solution for some $x: \norm{x}=1$. Substituting $K^P=k^PI_n$ and $\mathcal{L}_C=\mathcal{L}_R$, this equation becomes
\begin{align}
\begin{aligned}
0&= x^TQ(s)x =  \underbrace{ \frac{\gamma\delta}{k^P} x^T \mathcal{L}_R^3 x}_{a_0} \\
&+ s \underbrace{x^T\left[ \frac{\delta+\gamma}{k^P}\mathcal{L}_R^2 
+ \delta\mathcal{L}_R +\frac{\gamma\delta}{k^P}\mathcal{L}_R^2C^{-1} +K^V \right]x}_{a_1}  \\
&+   s^2 \underbrace{x^T\left[\frac{1}{k^P}\mathcal{L}_R + I_n +\frac{\gamma+\delta}{k^P}\mathcal{L}_RC^{-1} \right]x}_{a_2} \\
&+ s^3 \underbrace{\frac{1}{k^P} x^T C^{-1} x}_{a_3}.
\end{aligned}
\label{eq:x^TQx}
\end{align}
Clearly \eqref{eq:x^TQx} has one solution $s=0$ for $x=\frac{a}{\sqrt{n}}[1, \dots, 1]^T$, since this implies that $a_0=0$. The remaining solutions are stable if and only if the polynomial $a_1+ sa_2+s^2a_3=0$ is Hurwitz, which is equivalent to $a_i>0$ for $i=1,2,3$ by the Routh-Hurwitz stability criterion. For $x\ne \frac{a}{\sqrt{n}}[1, \dots, 1]^T$, we have that $a_0>0$, and thus $s=0$ cannot be a solution of \eqref{eq:x^TQx}. By the Routh-Hurwitz stability criterion, \eqref{eq:x^TQx} has only stable solutions if and only if $a_i>0$ for $i=0,1,2,3$ and $a_0a_3<a_1a_2$. Since this condition implies that $a_i>0$ for $i=1,2,3$, there is no need to check this second condition explicitly. Clearly $a_3>0$ since $K^{P^{-1}}$ and $C^{-1}$ are diagonal with positive elements. It is easily verified that $a_2>0$ if \eqref{eq:stability_A_sufficient_1} holds, since $\mathcal{L}_R\ge 0$. Similarly, $a_1>0$ if \eqref{eq:stability_A_sufficient_2} holds, since also $\mathcal{L}_R^2\ge0$ and $x^TK^Vx\ge \min_i K^V_i$. In order to assure that $a_0a_3<a_1a_2$, we need furthermore to upper bound $a_0a_3$. The following bound is easily verified
\begin{align*}
a_0a_3<\lambda_{\max}\left(\mathcal{L}_R^3\right) \frac{\gamma\delta}{k^{P^2}} \max_i C_i.
\end{align*} 
The previously obtained lower bounds on $a_1$ and $a_2$ give a lower bound $a_1a_2$. Thus \eqref{eq:stability_A_sufficient_3} is a sufficient condition for when $a_0a_3<a_1a_2$. 
\end{proof}

\section{Simulations}
\label{sec:simulations}
Simulations of an MTDC system were conducted using MATLAB. The MTDC  was modelled by \eqref{eq:hvdc_scalar}, with $u_i$ given by the distributed controller \eqref{eq:distributed_voltage_control}. The topology of the MTDC system is given by Figure \ref{fig:graph}. The capacities are assumed to be $C_i=123.79 \; \mu \text{F}$ for $i=1,2,3,4$, while the resistances are assumed to be $R_{12}=\Omega$, $R_{13}= \Omega$, $R_{24} \Omega$, $R_{34} 0.0065 \; \Omega$. The controller parameters were set to $K^P_i= 1 \; \Omega^{-1}$ for $i=1,2,3,4$, $\gamma = 0.005$ and $c_{ij} = R_{ij}^{-1} \; \Omega^{-1}$ for all $(i,j)\in \mathcal{E}$. Due to the long geographical distances between the DC buses, communication between neighboring nodes is assumed to be delayed with delay $\tau$. While the nominal system without time-delays is verified to be stable according to Theorem \ref{th:distributed_voltage_control},  time-delays might destabilize the system. It is thus of importance to study the effects of time-delays further. 
 The dynamics of the controller \eqref{eq:distributed_voltage_control} with time delays thus become

\begin{align}
\label{eq:distributed_voltage_control_delay}
\begin{aligned}
u_i &= K^P(\hat{V}_i(t) -V_i(t)) \\
\dot{\hat{V}}_i &= K^V_i(V^{\text{nom}}-V_i(t)) \\
- &\gamma \sum_{j\in \mathcal{N}_i} c_{ij} \left( (\hat{V}_i(t') -V_i(t')){-}(\hat{V}_j(t'){-}V_j(t')) \right) \\
\dot{\bar{V}}_i &= - K_i^V (V_i(t) - V_i^{\text{nom}}) -\delta \sum_{j\in \mathcal{N}_i} c_{ij} (\bar{V}_i(t') - \bar{V}_j(t')),
\end{aligned}
\end{align}
where $t'=t-\tau$. The injected currents are assumed to be initially given by $I^{\text{inj}}=[300,200,-100,-400]^T$ A, and the system is allowed to converge to the stationary solution. Since the injected currents satisfy $I_i^{\text{inj}}=0$, $u_i=0$ for $i=1,2,3,4$ by Theorem \ref{th:distributed_voltage_control}.
Then, at time $t=0$, the injected currents are changed due to changed power loads. The new injected currents are given by $I^{\text{inj}}=[  300, 200, -300, -400]^T$ A. The step response of the voltages $V_i$ and the controlled injected currents $u_i$ are shown in Figure \ref{fig:powersystems_sim_1}. The conservative voltage bounds guaranteed by Theorem, are indicated by $V_\text{min}$ and $V_\text{max}$. For the delay-free case, i.e., $\tau=0$ s, the voltages $V_i$ are restored close to their new stationary values within $2$ seconds. The controlled injected currents $u_i$ converge to their stationary values within $8$ seconds. The simulation with time delays $\tau=0.4$ s, show that the controller is robust to moderate time-delays. For a time delay of $\tau=0.5$ s, the system becomes unstable.

\setlength\fheight{3.5cm}
\setlength\fwidth{6.8cm}
\begin{figure*}[ht]
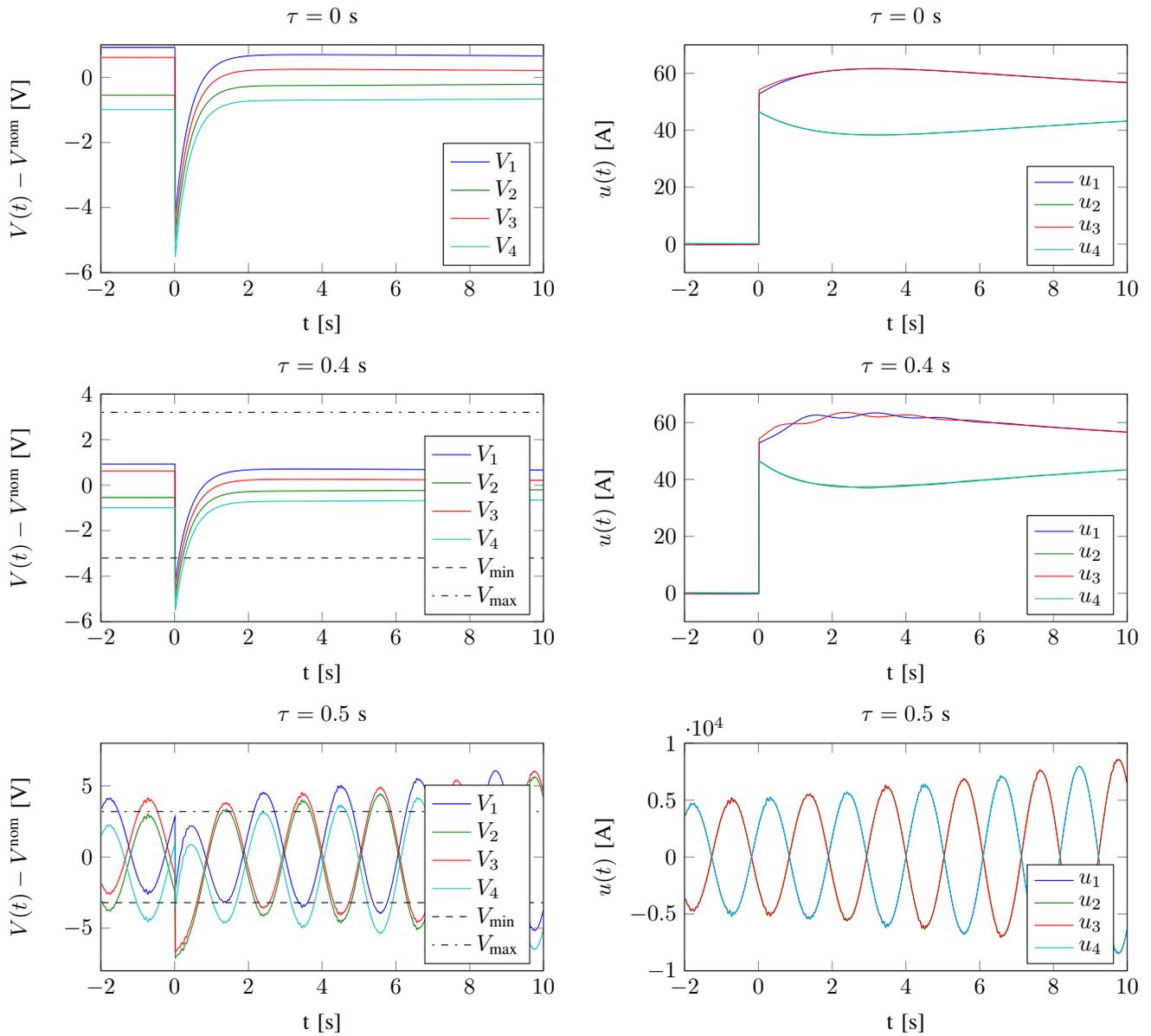

	\centering
$
\begin{array}{cc}
	\input{Simulations/V_tau=0gamma=0005.tikz} & \input{Simulations/u_tau=0gamma=0005.tikz} \\
	\input{Simulations/V_tau=02gamma=0005.tikz} & \input{Simulations/u_tau=02gamma=0005.tikz} \\
	\input{Simulations/V_tau=03gamma=0005.tikz} & \input{Simulations/u_tau=03gamma=0005.tikz}
\end{array}
$
\caption{The figure shows the voltages $V_i$ and the controlled injected currents $u_i$ of the DC buses for different time-delays $\tau$ on the communication links. The system model is given by \eqref{eq:hvdc_scalar}, and $u_i$ is given by the distributed controller \eqref{eq:distributed_voltage_control_delay}.}
\label{fig:powersystems_sim_1}
\end{figure*}

\section{Discussion and Conclusions}
\label{sec:discussion}
In this paper we have proposed a fully distributed controller for voltage and current control in MTDC networks. We show that under certain conditions, there exist controller parameters such that the closed-loop system is stabilized. We have shown that the proposed controller is able to maintain the voltage levels of the DC buses close to the nominal voltages, while at the same time, the global cost of the injected currents is asymptotically minimized. 

This paper lays the foundation for distributed control strategies for systems of interconnected AC and MTDC systems. Future work will in addition to the voltage dynamics of the MTDC system, also consider the dynamics of the connected AC systems. Interconnecting multiple asynchronous AC systems also enables novel control applications, for example automatic sharing of primary and secondary frequency control reserves. 

\bibliography{references}
\bibliographystyle{plain}
\end{document}